\documentclass{article}
\usepackage{amsmath}
\usepackage{mathtools}
\usepackage{amsfonts}
\usepackage{amssymb}
\usepackage{graphicx} % Required for inserting images
\usepackage{xcolor}
\usepackage{amsthm}
\usepackage{enumitem}
\usepackage{CJKutf8}
\usepackage{blkarray}
\usepackage{graphicx}
\theoremstyle{plain}
\newtheorem{theorem}{Theorem}[section]
\newtheorem*{openproblem}{Open Problem}
\newtheorem{lemma}[theorem]{Lemma}

\theoremstyle{definition} 
\newtheorem{example}{Example}[section]
\newtheorem{definition}[theorem]{Definition}
\usepackage{comment}
\theoremstyle{remark}
\newtheorem{remark}[theorem]{Remark}
\newcommand{\spans}{\operatorname{span}}
\newcommand{\Sym}{\operatorname{Sym}}
\newcommand{\SoS}{\operatorname{SoS}}
\newcommand{\RSoS}{\operatorname{RSoS}}

\usepackage{authblk}
\definecolor{lightpurple}{RGB}{110,0,190}
\newcommand{\dede}[1]{#1}
\newcommand{\cyan}[1]{#1}
\title{Noncommutative Rational Sums of Squares in Free Algebras }
\author[]{Sizhuo Yan} 
\author[]{Jianting Yang}
\author[]{Lihong Zhi}
\affil[]{State Key Laboratory of Mathematical Sciences, Academy of Mathematics and Systems Science, 
University of Chinese Academy of Science}
\date{}

\begin{document}

\maketitle
\begin{abstract}
 %    This paper introduces rational sums of squares for noncommutative polynomials in free algebras.  The main result gives, \dede{for noncommutative polynomials satisfying certain conditions,} an equivalent condition for this rational form: for every operator evaluation, the finite-dimensional compressions satisfying the prescribed conditions
 % determined by this evaluation are not negative definite. Several basic properties are also studied, including their relation to sums of Hermitian squares and their non-convexity.

%\begin{abstract}
This paper introduces rational sums of squares for symmetric noncommutative
polynomials in free algebras. Under the assumption that the polynomial admits a strictly positive scalar self-adjoint evaluation, our main result gives an operator-theoretic characterization of this class: a polynomial is a rational sum of squares   if and only if there exists a degree bound such that, for every self-adjoint operator evaluation satisfying a natural nondegeneracy condition, the compression of the evaluated polynomial to the associated finite-dimensional cyclic subspace is not negative definite. We also prove that sums of Hermitian squares  form a proper subset of rational sums of squares when the free algebra has at least two generators, whereas the two classes coincide for homogeneous polynomials. Finally, we show that the set of rational sums of squares is nonconvex when there are at least three generators, and that its complement is also nonconvex.
\end{abstract}

\section{Introduction}

% Noncommutative real algebraic geometry studies noncommutative polynomials and their evaluations at matrices and operators. It can be viewed as a noncommutative analog of real algebraic geometry and is closely related to operator theory, operator algebras, and noncommutative polynomial optimization\cyan{, see, for example,
% \cite{Schmudgen2009,Burg2016}.} \cyan{Related factorization results for noncommutative polynomials were established in~\cite{MR1815959}. For Positivstellensatz on constrained noncommutative semialgebraic sets, weighted sum-of-squares representations were developed in~\cite{MR2055751}, while the convex Positivstellensatz gives exact certificates, with degree bounds, on free spectrahedra~\cite{MR2935397}. For computational methods for noncommutative sums of Hermitian squares and related optimization problems, we refer the reader to~\cite{MR3041750,MR2580654}.
% } 

Noncommutative real algebraic geometry studies positivity and algebraic
certificates for noncommutative polynomials through their evaluations at
tuples of matrices and operators. It may be viewed as a noncommutative
analogue of real algebraic geometry and is closely connected to operator
theory, operator algebras, and noncommutative polynomial optimization; see,
for example,~\cite{Schmudgen2009,Burg2016}. Related factorization results for
noncommutative polynomials were established in~\cite{MR1815959}. For
positivity on constrained noncommutative semialgebraic sets, weighted
sum-of-squares representations were developed in~\cite{MR2055751}, while the
convex Positivstellensatz provides exact certificates, together with degree
bounds, on free spectrahedra~\cite{MR2935397}. For computational and
semidefinite programming methods for noncommutative sums of Hermitian squares
and related polynomial optimization problems, we refer the reader
to~\cite{MR3041750,MR2580654}.

%An important problem is to determine when positivity under all matrix evaluations admits an algebraic certificate.
% \cyan{One of the fundamental result} in this \cyan{field} is Helton's  sum-of-squares theorem. It states that a symmetric noncommutative polynomial that is positive semidefinite under all matrix evaluations is a sum of Hermitian squares~\cite{Helton2002positive}. This result shows a strong difference between the commutative and noncommutative cases. Noncommutative rational expressions are obtained from noncommutative polynomials by allowing inverses. In~\cite{klep2017regular}, positive regular noncommutative rational functions were studied, and it was proved that they are sums of Hermitian squares of regular rational functions. This can be viewed as a rational function analog of Helton's sum-of-squares theorem. 

One of the fundamental results in this field is Helton's sum-of-squares theorem, which states that a symmetric noncommutative polynomial that is
positive semidefinite under all self-adjoint matrix evaluations, in every
dimension, is a sum of Hermitian squares~\cite{Helton2002positive}. This
theorem reveals a striking contrast between the commutative and
noncommutative settings. Allowing inverses of noncommutative polynomials
leads to noncommutative rational expressions. In~\cite{klep2017regular}, it
was proved that every positive regular noncommutative rational function is a
sum of Hermitian squares of regular rational functions, providing a rational function analog of Helton's theorem. \cyan{For recent results on Positivstellensatz for noncommutative rational functions, we refer to~\cite{MR5030113,MR3750207,MR4308824}.}

% The present paper studies another rational form for noncommutative polynomials.  This form is motivated by \cyan{Stengle's Positivstellensatz~\cite{MR332747}} as well as an open problem of Klep and Schweighofer~\cite{IgorMarkus2007nicht}. In that problem, rationality is introduced by 
% multiplying a symmetric noncommutative polynomial on both sides by other noncommutative polynomials. This gives a different way to introduce rational forms into the study of noncommutative polynomials.

The present paper investigates a different notion of rationality for
noncommutative polynomials, motivated by Stengle's
Positivstellensatz~\cite{MR332747} and by an open problem of Klep and
Schweighofer~\cite{IgorMarkus2007nicht}. Rather than introducing inverses,
this approach uses noncommutative polynomial multipliers: a symmetric
polynomial is multiplied on the left and right by other noncommutative
polynomials, and the resulting expressions are required to satisfy a
sum-of-squares condition. This multiplier-based construction provides an
alternative way to formulate rational positivity certificates entirely
within the free polynomial algebra.

The main contributions of this paper are as follows.

\begin{enumerate}
    \item We introduce the class $\RSoS$ and establish its first structural
    properties. For every free algebra with at least two generators, we
    construct an explicit polynomial showing that
    \[
        \SoS \subsetneq \RSoS.
    \]
    We further prove that the leading homogeneous part of every $\RSoS$
    polynomial belongs to $\SoS$. Consequently, a symmetric homogeneous
    noncommutative polynomial belongs to $\RSoS$ if and only if it belongs
    to $\SoS$.

    \item Let $f=f^*$ and assume that there exists a scalar self-adjoint
    tuple $a$ such that $f(a)>0$. We prove an operator-theoretic
    characterization of $f\in\RSoS$. More precisely, $f$ is a rational sum-of-squares if and only if there exists a degree bound $d$ such that, for every
self-adjoint operator evaluation satisfying a natural nondegeneracy condition,
% the compression of the evaluated polynomial to the associated
% finite-dimensional cyclic subspace is not negative definite.   
%     if and only if there exists a positive integer $d$ such that, for every nontrivial separable 
% $k$-Hilbert space $\mathcal{H}$ with a fixed unit vector $e_1$ and every self-adjoint tuple
% $\underline{A}\in B(\mathcal{H})^m$, 
% if $g(\underline A)e_1\neq0$ for every nonzero $g\in k\langle\underline X\rangle$ with $\deg(g)\le d$, 
then the compression
 of $ f(\underline{A})$ on $W_{\underline{A}}$ is not negative definite, where  
\[
    W_{\underline{A}}
    := \spans\{\,\omega(\underline{A})e_1 \mid \deg(\omega)\le d,  ~~e_1~\text{is a fixed unit vector} \}.
\]
    The proof combines closedness of
    degree-truncated multiplier cones, strict separation of convex cones,
    and a truncated GNS construction.

    \item We show that $\RSoS$ is not convex when the free algebra has at
    least three generators. We also prove that its complement 
    is not convex whenever the free algebra has at least one generator.
    This contrasts with the commutative setting, where rational sums of
    squares coincide with nonnegative polynomials and therefore form a
    convex cone.
\end{enumerate}

\subsection{Notation}

 The notation used throughout this paper is summarized as follows. Here
  $\mathbb{N}$ denotes the set of natural numbers, $\mathbb{N}^{+}$ denotes the set of positive integers and $k$ denotes either the field of real numbers or the field of complex numbers. 
 In this work, the main objects of study are free algebras of noncommutative polynomials. Let $\underline{X} = (X_1, X_2, \ldots, X_m)$ be a tuple of noncommutative variables,  where $m$ is a positive integer. The set of all words (monomials) generated by $\underline{X}$ is denoted by $\langle\underline{X}\rangle$. The notation $k\langle\underline{X}\rangle$ denotes the free associative algebra over $k$ generated by $\underline{X}$, whose elements are noncommutative polynomials.
 A noncommutative polynomial in $k\langle \underline{X}\rangle$ has the following form:
 \[f=\sum_{i=1}^{N}\alpha_{i} \omega_i,~ \text{where} ~\omega_i\in \langle \underline{X} \rangle,~\alpha_i\in k,~N\in \mathbb{N}.\]
  The degree of a word (monomial) $\omega$ is its length, denoted by $\mathrm{deg}(\omega)$.  The degree of a noncommutative polynomial $f$ is the  maximum length of a word appearing in $f$, denoted by $\deg(f)$. \dede{The notation $k\langle\underline{X}\rangle^d$ denotes the set of all noncommutative polynomials whose degree is at most $d$}\cyan{, i.e.,}
 \[
    k\langle \underline{X}\rangle^{ d}
    :=\{f\in k\langle \underline{X}\rangle\mid \deg(f)\le d\}.
\]
 For a monomial $\omega=X_{i_1} X_{i_2} \ldots X_{i_l}$, where $ l \in \mathbb{N}$ and $i_1, \ldots, i_l \in \{1, \ldots, m\}$, its involution is defined by reversing the order of variables, i.e., $\omega^*=X_{i_l} \cdots X_{i_2}  X_{i_1}$. Then the involution of a noncommutative polynomial is defined by
 \[f^{*}=\sum_{i=1}^{N}\overline{\alpha_{i}} \omega_i^{*},~ \text{where} ~\omega_i\in \langle \underline{X} \rangle,~\alpha_i\in k,~N\in \mathbb{N},\]
 where $\overline{\alpha_i}$ is the complex conjugate of $\alpha_i$.
 A noncommutative polynomial $f$ is called symmetric if $f=f^{*}$, and the set of symmetric noncommutative polynomials is denoted by $\Sym k\langle\underline{X}\rangle$.

 A noncommutative polynomial $f$  admits a sum-of-squares (SoS)  form if there exist finitely many noncommutative polynomials $g_1, \ldots, g_r \in k\langle \underline{X}\rangle$ such that
 \[
 f = \sum_{i=1}^r g_i^{*} g_i.
 \]
 For convenience, the set of all noncommutative polynomials  that admit a sum-of-squares form is denoted by $\SoS$.

 Let $\mathcal{H}$ be a  separable $k$-Hilbert space, and let $\underline{A} = (A_1, A_2,\ldots, A_m) \in \mathrm{B}(\mathcal{H})^{m}$ denote a tuple of self-adjoint operators. The evaluation of a noncommutative polynomial $f \in k\langle \underline{X} \rangle$ at $\underline{A}\in \mathrm{B}(\mathcal{H})^{m}$ is denoted by $f(\underline{A})$. % example 

   \subsection{Background}

    In \cite{IgorMarkus2007nicht}, Klep and Schweighofer posed the following open problem.
     \begin{openproblem}\label{openproblem}
 	\cite{IgorMarkus2007nicht} Given a noncommutative symmetric polynomial  $f$, are the following two conditions  equivalent?
 	\begin{enumerate}[label=(\arabic*)]
        \item For any  nontrivial $k$-Hilbert space $\mathcal{H}$ ($\mathcal{H}\neq0$) and every tuple $\underline{A} = (A_1, A_2,\ldots, A_m) \in \mathrm{B}(\mathcal{H})^{m}$, $f(\underline{A})$ is not negative semidefinite;
        \item There exist $r \in \mathbb{N} $ and  noncommutative polynomials $g_1,\dots g_r\in k\langle\underline{X}\rangle$  such that
 		\begin{equation}\label{con2}
 			\sum_{i=1}^r{g_i^* f g_i}\in 1+\SoS.
 		\end{equation}
			
 	\end{enumerate}
 \end{openproblem}
 This problem was answered negatively in \cite{MR4931501} by the following
example.
    \begin{example}\label{counterexample}
 	\begin{equation}\label{counterex}
 		f_0(X_1,X_2)=X_1 X_2^2 X_1-X_2 X_1^2 X_2+1.
 	\end{equation}
 \end{example}
 More precisely, \(f_0\) satisfies Condition~(1) in the above open problem, but
does not satisfy Condition~(2).
   This raises the natural question of how to modify the above conditions to establish their equivalence. This question will be resolved in Section~3. As noted in \cite{IgorMarkus2007nicht}, ``an affirmative answer to this problem would give a noncommutative analog of Artin’s solution to Hilbert’s 17th problem". Guided by this analogy, we introduce the notion of rational sums of squares for noncommutative polynomials.

    Let $f$ be a noncommutative symmetric polynomial,   and define the set,
   \[\Sigma_f {\coloneq}\{\sum_{i=1}^{r}g_i^*fg_i | g_i\in k \langle\underline{X}\rangle , r\in \mathbb{N}^{+}\}.\]
   Note that the polynomials $g_1,\ldots,g_r$ are not assumed to be nonzero here. Although only the case where polynomials $g_1,\ldots,g_r$ are nonzero is worth considering, we do not impose this restriction here, so that $\Sigma_f$ forms a cone and the subsequent proof is simplified.

   Next, we consider the basic properties of this cone.  For a nonzero polynomial $f$ and a finite family of nonzero polynomials $g_1,\ldots,g_r$,
   \[\sum_{i=1}^{r}g_i^*fg_i \neq 0.\]
   Indeed, since the polynomials $g_1,\ldots,g_r$ are nonzero, there exists a monomial $\omega$ of degree
\[
\deg(\omega) =\max_{1\leq i\leq r}\deg (g_i).
\]
Let $a_i$ be the coefficient of $\omega$ in $g_i$. Then at least one $a_i$ is nonzero. Similarly, choose a monomial $u$ of degree $\deg( f)$, and let $c\neq 0$ be its coefficient in $f$. By degree considerations, the coefficient of $\omega^{*}u\omega$ in
$
\sum_{i=1}^r g_i^*fg_i
$
is
\[
c\sum_{i=1}^r \cyan{|a_i|}^2\neq 0.
\]
Therefore, $\sum_{i=1}^r g_i^*fg_i$ is nonzero. It follows that $\Sigma_f$ is a pointed cone. Indeed, if there exist \cyan{$g_1,\ldots,g_{r_1}$ and $h_1,\ldots,h_{r_2}$}, such that
\[\sum_{i=1}^{r_1} g_i^*fg_i=-\sum_{i=1}^{r_2} h_i^*fh_i,\]
\cyan{then}
\[\sum_{i=1}^{r_1} g_i^*fg_i+\sum_{i=1}^{r_2} h_i^*fh_i=0.\]
By the preceding result, this implies that the $g_1,\ldots,g_r$ and $h_1,\ldots,h_r$ are all zero, and hence
\[\sum_{i=1}^{r_1} g_i^*fg_i=\sum_{i=1}^{r_2} h_i^*fh_i=0.\]

  \begin{definition}
  Let $f$ be a noncommutative symmetric polynomial,   and define the set,
   \[\Sigma_f {\coloneq}\{\sum_{i=1}^{r}g_i^*fg_i | g_i\in k \langle\underline{X}\rangle , r\in \mathbb{N}^{+}\}.\]
      For a noncommutative symmetric polynomial $f$,  we say that $f $ admits a \emph{rational sum-of-squares form}, if either $f=0$ or
      \[\Sigma_f\cap \SoS \neq \{0\}.\]Equivalently, \cyan{for nonzero symmetric polynomial $f$, $f$ admits a rational sum-of-squares form if and only if} there exist $r \in \mathbb{N}^{+} $ and  noncommutative polynomials $g_1,\dots g_r\in k\langle\underline{X}\rangle$  such that 	\[0\neq\sum_{i=1}^r{g_i^* f g_i}\in \SoS.\] 
   Moreover, the set of all \cyan{noncommutative polynomials that admit a rational sum-of-squares form} is denoted by $\RSoS$. 
      \end{definition}

\cyan{We remark that the notion of a rational sum-of-squares form considered here differs from that used in some previous works, such as  \cite{MR3414573}, where ``rational sums of Hermitian squares'' refers to sums of Hermitian squares with rational coefficients.}

    \section{Noncommutative RSoS Polynomial}

%    In the preceding section, we introduced noncommutative \(\RSoS\) polynomials. Since one may take the multiplier to be \(1\), every noncommutative \(\SoS\) polynomial is also an \(\RSoS\) polynomial. We now show that this inclusion is proper. More precisely, we construct an explicit noncommutative polynomial \(P\) such that \(P\in\RSoS\) but \(P\notin\SoS\). Hence \(\RSoS\) is a genuine enlargement of \(\SoS\).

 In the previous section, we introduced the definition of a noncommutative $\RSoS$ polynomial. Clearly, every noncommutative $\SoS$ polynomial is also a noncommutative $\RSoS$ polynomial. 
We begin by demonstrating that the inclusion is strict (i.e., $\SoS \subsetneq \RSoS$), hence the definition of $\RSoS$ is  nontrivial. Indeed, we construct an explicit noncommutative polynomial $P$ that is $\RSoS$ but not $\SoS$.

\begin{comment}
  \begin{example}\label{ex:RSOS-not-SOS}
Assume \(m\ge 2\), and write \(X=X_1\) and \(Y=X_2\). Define
\[
    P(X,Y)
    :=2Y^2X^2Y^2+Y^2X^2Y+YX^2Y^2+X^2
    \in k\langle X_1,\ldots,X_m\rangle .
\]
\end{example}

\begin{theorem}\label{thm:SOS-strict-RSOS}
For every free algebra \(k\langle X_1,\ldots,X_m\rangle\) with \(m\ge 2\),
one has
\[
    \SoS \subsetneq \RSoS.
\]
More precisely, the polynomial \(P\) in Example~\ref{ex:RSOS-not-SOS}
satisfies
\[
    P\in \RSoS\setminus \SoS .
\]
\end{theorem}  
\end{comment}

  \begin{example}\label{example}
 For convenience, let \cyan{$X=X_1,Y=X_2$}, and \cyan{define}
       $$P(X,Y)\coloneq 2Y^2X^2Y^2+Y^2X^2Y+YX^2Y^2+X^2  \in k\langle X,Y\rangle$$ 
  \end{example}
  \begin{theorem}\label{definition}
  For every free algebra \(k\langle X_1,\ldots,X_m\rangle\) with \(m\ge 2\),
   $\SoS$ is strictly contained in $\RSoS$, i.e.,
     \[\SoS\subsetneq \RSoS.\]
    \end{theorem}

\begin{comment}
Before the proof, we recall the Gram-matrix representation for
noncommutative polynomials. By \cite[Lemma~2.1]{Helton2002positive}, for every
symmetric polynomial \(f\in k\langle \underline X\rangle\), there exists a
self-adjoint matrix \(M_f\) such that
\[
    f(\underline X)=V^d(\underline X)^* M_f V^d(\underline X),
\]
where \(d=\lceil \deg(f)/2\rceil\), and \(V^d(\underline X)\) is the column
vector of all monomials of degree at most \(d\). Such a matrix \(M_f\) is not
unique; we call any matrix satisfying the above identity a coefficient matrix
of \(f\). Its rows and columns are indexed by the monomials appearing in
\(V^d(\underline X)\).

Equivalently, for every monomial \(\omega\), its coefficient in \(f\) is
\[
    \sum_{\omega=\omega_1^*\omega_2} M_f(\omega_1,\omega_2),
\]
where the sum runs over all monomials
\(\omega_1,\omega_2\) of degree at most \(d\). By
\cite[Theorem~2.1]{McCullough2005}, a symmetric noncommutative polynomial
\(f\) belongs to \(\SoS\) if and only if it admits a positive semidefinite
coefficient matrix.
\end{comment}

 Before the proof, the following important concept is needed. By \cite[Lemma 2.1]{Helton2002positive},  for every noncommutative symmetric polynomial $f$ there exists a self-adjoint matrix $M_f$, such that
		\[f(\underline{X})=V^{d}(\underline{X})^* M_f V^{d}(\underline{X}).\]
	 where $d=\lceil \deg(f)/2 \rceil$, and  $V^{d}(\underline{X})$ is a vector whose entries consist of all monomials of degree at most $d$. The matrices satisfying the above conditions are not unique; any of them is called a coefficient matrix of $f$. \cyan{For notational convenience, the rows and columns of $M_f$ are indexed by the monomials in $V^{d}(\underline{X})$}.
    For a monomial $\omega$ in $f$ \cyan{and any coefficient matrix $M_f$},  there exist corresponding entries whose indices $(\omega_1,\omega_2)$ satisfies $\omega=\omega_1^*\omega_2$ in $M_f$, and  the coefficient of $\omega$ in $f$ equals  $\sum_{\omega=\omega_1^*\omega_2}M_f(\omega_1,\omega_2)$.
    By \cite[Theorem 2.1]{McCullough2005},
	 a noncommutative polynomial $f$ belongs to $\SoS$ if and only if there exists a positive semidefinite coefficient matrix $M_f$.

\begin{proof}[\cyan{Proof of Theorem~\ref{definition}}]

\begin{comment}
    The proof of Theorem~\ref{definition} is based on the key fact that
the polynomial \(P(X,Y)\) in Example~\ref{example} satisfies
\(P\in\RSoS\setminus\SoS\).
We first show that \(P\notin\SoS\). Since \(\deg(P)=6\), any coefficient
matrix for \(P\) may be indexed by the monomials of degree at most \(3\).
Let \(M_P\) be an arbitrary self-adjoint coefficient matrix of \(P\). For the
monomial \(Y^2X^2Y\), the only decompositions
\(\omega_1^*\omega_2=Y^2X^2Y\), with
\(\deg(\omega_1),\deg(\omega_2)\le 3\), are
\[
    (XY^2)^*(XY)=Y^2X^2Y,
    \qquad
    (Y^2)^*(X^2Y)=Y^2X^2Y .
\]
Similarly, the monomial \(YX^2Y^2\) arises only from
\[
    (XY)^*(XY^2)=YX^2Y^2,
    \qquad
    (X^2Y)^*(Y^2)=YX^2Y^2 .
\]
Therefore the corresponding entries of \(M_P\) satisfy
\[
\begin{aligned}
    M_P(XY^2,XY)+M_P(Y^2,X^2Y) &= 1,\\
    M_P(XY,XY^2)+M_P(X^2Y,Y^2) &= 1.
\end{aligned}
\]
Since \(M_P\) is self-adjoint, we also have
\[
\begin{aligned}
    M_P(XY^2,XY) &= \overline{M_P(XY,XY^2)},\\
    M_P(Y^2,X^2Y) &= \overline{M_P(X^2Y,Y^2)} .
\end{aligned}
\]
Equivalently,
\[
\begin{aligned}
    &M_P(XY^2,XY)+M_P(Y^2,X^2Y) = 1,\\
    &M_P(XY,XY^2)+M_P(X^2Y,Y^2) = 1,\\
    &M_P(XY^2,XY) = \overline{M_P(XY,XY^2)},\\
    &M_P(Y^2,X^2Y) = \overline{M_P(X^2Y,Y^2)} .
\end{aligned}
\]
    
\end{comment}

    This theorem will be proved by using the following fact. The noncommutative polynomial $P(X,Y)$ belongs to $\RSoS$, but does not belong to $\SoS$.
    
    It is first shown that the noncommutative polynomial $P$ is not a noncommutative $\SoS$ polynomial by proving that its coefficient matrices cannot be positive semidefinite. The degree of the noncommutative polynomial $P$ is 6. Hence, the indices of its coefficient matrix consist of all monomials of degree at most 3. The corresponding entries of monomial $Y^2X^2Y$ are indexed by $(XY^2, XY)$ and $(Y^2,X^2Y)$. The corresponding entries of monomial $YX^2Y^2$ are indexed by $(XY,XY^2)$ and $(X^2Y,Y^2)$, and they satisfy the following conditions.
     \begin{eqnarray*}
       &&M_P(XY^2, XY)+M_P(Y^2,X^2Y)=1,\\ 
       &&M_P(XY, XY^2)+M_P(X^2Y,Y^2)=1,\\
       &&M_P(XY^2, XY)=\overline{M_P(XY, XY^2)}, \\
       &&M_P(Y^2,X^2Y)=\overline{M_P(X^2Y,Y^2)}. 
     \end{eqnarray*}
  \begin{comment}
    For the coefficient matrix of noncommutative polynomial $P$, \cyan{there are only two entries of the coefficient matrix related  to the term $Y^2X^2Y$ and $YX^2Y^2$  respectively, which are} \[\{(XY^2, XY),(XY,XY^2)\} \text{ or }\{(Y^2,X^2Y),(X^2Y,Y^2)\},\]
        \cyan{That is, if the rows and columns of $M_f$ are indexed by elements in $V^{3}(\underline{X})$ (that is all monomials of degree at most $3$), then 
    \[M_f(XY^2,XY)+M_f(Y^2,X^2Y) \text{ and } M_f(X^2Y,Y^2)+M_f(XY,XY^2)\]
    should  exactly equal the coefficient of $Y^2X^2Y$ and  $YX^2Y^2$   in $P$ respectively. \jianting{rewrite  what is indexed by }}
\end{comment}
     Similarly, the corresponding entries of monomial $X^2$ are indexed by $(X,X)$, $(1,X^2)$ and $(X^2,1)$, and they satisfy the following conditions.
     \begin{eqnarray*}
     && M_P(X,X)+M_P(1,X^2)+M_P(X^2,1)=1,\\
         && M_P(1,X^2)=\overline{M_P(X^2,1)}.
     \end{eqnarray*}
 The corresponding entries of monomial $Y^2X^2Y^2$ are indexed by $(XY^2,XY^2)$, and $M_P(XY^2,XY^2)=2$.

Similarly, we have the following conditions for $M_P$:
 \begin{eqnarray}
    && M_P(X^2Y,X^2Y)=0, \nonumber\\
    && M_P(XY,XY)+ M_P(X^2Y,Y)+M_P(Y,X^2Y)=0. \label{eq:linear4Lemma2.1}
 \end{eqnarray}
Now consider the following principal submatrix of the coefficient matrix $M_P$.
    
	\[
    \scalebox{0.75}{\bordermatrix{   &XY&YX&Y^2&X^3&X^2Y&XYX&YX^2&XY^2\cr
		XY  & & & & & & & & M_P(XY,XY^2) \cr
		YX  & & & & & & & &  \cr
		Y^2  & & & & &M_P(Y^2,X^2Y) & & &  \cr
		X^3  & & & & & & & & \cr
		X^2Y & & &M_P(X^2Y,Y^2) & & 0& & &  \cr
		XYX  & & & & & & & & \cr
		YX^2 & & & & & & & &  \cr
		XY^2 &M_P(XY^2,XY)& & & & & & &2  \cr
	 }} \]
    If the coefficient matrix $M_P$ is positive semidefinite, then this submatrix of $M_P$ is also positive semidefinite. Hence  the condition $M_P(X^2Y,X^2Y)=0$ implies   $M_P(X^2Y,Y)=0$ and  $M_P(Y,X^2Y)=0$. Therefore, by \eqref{eq:linear4Lemma2.1}, $M_P(XY,XY)=0$ and   
    the entries in the submatrix satisfy the following condition:
    \[M_P(XY^2, XY)=M_P(Y^2,X^2Y)=M_P(XY, XY^2)=M_P(X^2Y,Y^2)=0,\] 
 which contradicts the constraint that their sum is $2$.
    So the coefficient matrix cannot be positive semidefinite. Therefore, the noncommutative polynomial $P$ is not a noncommutative $\SoS$ polynomial.

\begin{comment}
  We next prove that \(P\in\RSoS\). Take
\[
    g_1=1,\qquad g_2=Y .
\]
Since \(Y^*=Y\), we have
\[
    \sum_{i=1}^2 g_i^*Pg_i=P+YPY .
\]
A direct computation gives the following sum-of-squares decomposition:
\[
\begin{aligned}
P+YPY
={}&(XY^3)^*(XY^3)
 +(XY^3+XY^2)^*(XY^3+XY^2)  \\
&+(XY^2+XY)^*(XY^2+XY)+X^*X .
\end{aligned}
\]
Thus \(\sum_{i=1}^2 g_i^*Pg_i\in\SoS\), and hence \(P\in\RSoS\).
Since we have already shown that \(P\notin\SoS\), it follows that
\(P\in\RSoS\setminus\SoS\). Therefore,
\[
    \SoS\subsetneq\RSoS .
\]
  %\end{proof}
The next result shows that this strict enlargement disappears in the
homogeneous case.
\end{comment} 

   \cyan{Now we prove that $P \in \RSoS$}. Let 
   \[g_1=1\text{ , }g_2=Y\]
     \begin{eqnarray*}
     \sum_{i=1}^{2}g_i^*Pg_i&=& P+YPY\\
     &=&2Y^3X^2Y^3+Y^3X^2Y^2+Y^2X^2Y^3+YX^2Y\\
     &&+2Y^2X^2Y^2+Y^2X^2Y+YX^2Y^2+X^2\\
    &=&(XY^3)^{*}XY^3+(XY^3+XY^2)^{*}(XY^3+XY^2)\\
     &&+(XY^2+XY)^{*}(XY^2+XY)+X^2
 \end{eqnarray*}
     Hence, the noncommutative polynomial $P$ is a  noncommutative $\RSoS$ polynomial. The set of noncommutative $\SoS$ polynomials is a proper subset of
      the set of noncommutative $\RSoS$ polynomials.
     \end{proof}
     In the homogeneous case, the set of noncommutative $\RSoS$ polynomials and the set of noncommutative $\SoS$ polynomials are equal. The following \cyan{theorem} gives the relevant results.

\begin{comment}

\begin{theorem}\label{thm:homogeneous:RSOS=SOS}
Let \(f\in k\langle \underline X\rangle\) be a homogeneous symmetric
noncommutative polynomial. Then
\[
    f\in\SoS \quad \Longleftrightarrow \quad f\in\RSoS .
\]
\end{theorem}

Before proving the theorem, we record a necessary condition for membership in
\(\RSoS\). For a nonzero noncommutative polynomial \(f\), let
\(\operatorname{ld}(f)\) denote its leading homogeneous part, namely the sum of
all terms of \(f\) of degree \(\deg(f)\).

\begin{lemma}[Necessary condition]\label{lemma leading polynomial}
Let \(f\in k\langle \underline X\rangle\) be a symmetric noncommutative
polynomial. If \(f\in\RSoS\), then
\[
    \operatorname{ld}(f)\in\SoS .
\]
\end{lemma}
    
\end{comment}

 \begin{theorem}\label{thm:homogeneous:RSOS=SOS}
           If $f \in k\langle X_1,\ldots,X_m\rangle $ is a noncommutative homogeneous polynomial, then $f$ belongs to $\SoS$ if and only if $f$ belongs to $\RSoS$.
     \end{theorem}
         Before proving the theorem, we need to introduce the following lemma regarding the necessary condition for membership in  $\RSoS$.         
         For a noncommutative polynomial $f$, its \emph{leading polynomial} $\operatorname{ld}(f)$ is defined as the sum of all terms of $f$ with degree $\deg(f)$. This yields the following necessary condition. 
     \begin{lemma}[Necessary condition]\label{lemma leading polynomial}
         Let  $f \in k\langle X_1,\ldots,X_m\rangle $  be a noncommutative  polynomial. If $f$ belongs to $\RSoS$, then its leading polynomial $\operatorname{ld}(f)$  belongs to $\SoS$.
     \end{lemma}
      \begin{proof}
            If $f$ is a noncommutative $\RSoS$ polynomial, there exist finitely many  noncommutative polynomials $g_1,g_2,\ldots,g_{r}$ such that 
            \[F=\sum_{i=1}^{r}g_i^{*}fg_i\in \SoS\]
 Hence, there exists a coefficient matrix $M_{F}$ of the noncommutative polynomial $F$ which is positive semidefinite.
        Let $D_{max}=\max \{\deg(g_1),\deg(g_2),\ldots\deg(g_r)\}$; then there exists a monomial $\omega$ in $g_i$ for some $1\leq i \leq r$, such that $\deg(\omega)=D_{max}$.  Let \(c_{i,\omega}\) be the coefficient of
\(\omega\) in \(g_i\), and 
\[
    C_\omega:=\sum_{i=1}^r |c_{i,\omega}|^2>0 .
\]
        
        Let $M^{sub}_{F}$ be the principal submatrix of $M_{F}$ consisting of entries indexed by monomials ending with $\omega$ and having degree $\deg(F)/2$. Then $M^{sub}_{F}$ is positive semidefinite. The terms in $F$ corresponding to the submatrix $M^{sub}_{F}$ can only have the form
        \[\alpha_{\nu}\omega^{*}\nu \omega \]
      where $\alpha_{\nu}\nu$, with $\alpha_\nu\in k$ and $\nu \in  \langle \underline{X} \rangle$, is a term of $f$ with $\deg(\nu)=\deg(f).$
        The following noncommutative polynomial can be obtained by summing all terms of the above form.
        \[f_{\omega}=\omega^{*}\operatorname{ld}(f)\omega.\]
        \dede{The above form} is a noncommutative homogeneous polynomial, and it has an even degree; its coefficient matrix is denoted by $M_{f_{\omega}}$. Let $M^{sub}_{f_{\omega}}$ be the principal submatrix of $M_{f_{\omega}}$ consisting of entries indexed by monomials ending with $\omega$ and having degree $\deg(f_{\omega})/2=\deg(F)/2$. \dede{The matrix $M^{sub}_{f_{\omega}}$} is unique due to the degree constraints on monomials used as indices. Then
        \[C_\omega M^{sub}_{f_{\omega}}=M^{sub}_{F}\succeq 0,\]
        and the entries in $M_{f_{\omega}}$ but not in $M^{sub}_{f_{\omega}}$ are all zero. So the coefficient matrix $M_{f_{\omega}}$ is positive semidefinite; thus, the noncommutative polynomial $f_{\omega}$ belongs to $\SoS$. There exist finitely many noncommutative polynomials $h_1,h_2,\ldots,h_l$ such that
        \[f_{\omega}=\sum^{l}_{j=1}h_j^*h_j.\]
        Since all entries of $M_{f_{\omega}}$ outside
$M^{sub}_{f_{\omega}}$ are zero,  all noncommutative polynomials $h_1,h_2,\ldots,h_l$ are homogeneous, and there exist noncommutative polynomials $h^{'}_1,h^{'}_2,\ldots,h^{'}_l$ such that
        \[h_j= h^{'}_j\omega,~\text{for all}~1\leq j \leq l,\]
        \[\operatorname{ld}(f)=\sum_{j=1}^{l}h_j^{'*}h_j^{'}.\]
        \end{proof}

\begin{remark}
    For Example~\ref{counterexample} $f_0(X_1,X_2)=X_1 X_2^2 X_1-X_2 X_1^2 X_2+1$, its leading polynomial is $X_1 X_2^2 X_1-X_2 X_1^2 X_2$. Since there does not exist a positive semidefinite coefficient matrix for its leading polynomial, it follows that $X_1 X_2^2 X_1-X_2 X_1^2 X_2$ does not belong to $\SoS$. By Lemma~\ref{lemma leading polynomial}, $f_0$ does not belong to $\RSoS$, thus $f_0$ does not satisfy Condition (2) in the open problem. 
 \end{remark}
Now we are ready to prove Theorem~\ref{thm:homogeneous:RSOS=SOS}.
   \begin{proof}[Proof of Theorem~\ref{thm:homogeneous:RSOS=SOS}]
         $f\in \SoS \Rightarrow f\in\RSoS$ is clear.
          Conversely, suppose $f$ is homogeneous. Then the leading polynomial $\operatorname{ld}(f)=f$. Hence, by Lemma~\ref{lemma leading polynomial}, $f$ belongs to $\SoS$.
     \end{proof}
      
    % \section{The non-Archimedean  Nirgendsnegativsemi-definitheitsstellensatz}

    \section{ A Compression Characterization of Noncommutative Rational Sums of Squares}
 % Consider the following condition: for a noncommutative symmetric polynomial $f$, there exist $r \in \mathbb{N} $ and noncommutative polynomials $g_1,\dots g_r\in k\langle\underline{X}\rangle$  such that
	% 		\begin{equation*}\label{con2}
	% 			\sum_{i=1}^r{g_i^* f g_i}\in{\rm SoS}.
	% 		\end{equation*}
 %  The finite collection of noncommutative polynomials necessarily admits a uniform upper bound on their degrees. So the above condition is equivalent to the condition that there exist $r,d\in \mathbb{N} $ and noncommutative polynomials $g_1,\dots g_r\in k\langle\underline{X}\rangle$  such that
 %  \[\operatorname{max}\{\deg(g_1),\deg(g_2),\ldots,\deg(g_r)\}\leq d,\]
	% 		\begin{equation*}\label{con2}
	% 			\sum_{i=1}^r{g_i^* f g_i}\in{\rm SoS}.
	% 		\end{equation*}
 %  Therefore, there exists a natural classification of noncommutative $\RSoS$ polynomials parameterized by \cyan{ the degree bound of $g_i$}.
 %  Hence the following theorem is proposed by \cyan{introducing} the upper bound of degrees of finite noncommutative polynomials $g_1,g_2,\ldots,g_r$.

Consider the following condition for a symmetric noncommutative polynomial $f\in  k\langle X_1,\ldots,X_m\rangle=k\langle\underline{X}\rangle$:
there exist $r\in\mathbb{N}$ and $g_1,\ldots,g_r\in  k\langle\underline{X}\rangle$
such that
\begin{equation}\label{con2}
    0\neq\sum_{i=1}^r g_i^* f g_i \in {\rm SoS}.
\end{equation}
Since the noncommutative polynomials  $g_1,\ldots,g_r$ are finite in number, their degrees are
bounded. Thus \eqref{con2} naturally leads to a degree filtration of
noncommutative $\RSoS$ polynomials, parameterized by an upper bound $d$ on
$\deg(g_i)$. The following theorem is formulated in terms of this degree bound.

\begin{theorem}\label{main}
 
 Let $f=f^*\in k\langle \underline{X}\rangle$ be a noncommutative symmetric polynomial. Assume that there exists a scalar self-adjoint tuple 
  $\underline{a}\in \mathrm{M}_1(k)^m$ such that $f(\underline{a})>0$. 
% $\underline{a}\in k^m$ \cyan{(regarded as $\mathrm{M_1}(k)^m$)}, such that $f(\underline{a})>0$, 
 Then the following conditions are equivalent.
\begin{enumerate}[label=(\arabic*)]

    \item 
There exists a positive integer $d$ such that, for every nontrivial separable 
$k$-Hilbert space $\mathcal{H}$ with a fixed unit vector $e_1$ and every self-adjoint tuple
$\underline{A}\in B(\mathcal{H})^m$, if $g(\underline A)e_1\neq0$ for every nonzero $g\in k\langle\underline X\rangle$ with $\deg(g)\le d$, then the compression
 of $ f(\underline{A})$ on $W_{\underline{A}}$ is not negative definite, where 
\[
    W_{\underline{A}}
    := \spans\{\,\omega(\underline{A})e_1 \mid \deg(\omega)\le d\,\}.
\]

  % \item $f$ is a noncommutative $\RSoS$ polynomial,  more precisely, there exist $r \in \mathbb{N} $ and  noncommutative polynomials $g_1,\dots g_r\in k\langle\underline{X}\rangle$  such that 
		% 	\begin{equation*}
		% 		\sum_{i=1}^r{g_i^* f g_i}\in{\rm SoS}.
		% 	\end{equation*}
  %       \end{enumerate}

    \item There exist
$r\in\mathbb{N}^{+}$ and $g_1,\ldots,g_r\in k\langle\underline{X}\rangle$
such that
\[
    0\neq\sum_{i=1}^r g_i^* f g_i \in {\rm SoS}.
\]

\end{enumerate}

\end{theorem}

\begin{remark}
Condition~(2) in Theorem~\ref{main} differs from the corresponding condition
in the open problem.  This modification is motivated by Artin's solution to
Hilbert's seventeenth problem.  In the commutative case, Artin proved that for
a real polynomial \(f\in \mathbb{R}[X_1,\ldots,X_n]\), the condition \(f(x)\geq 0\) for all
\(x\in\mathbb{R}^n\) is equivalent to \(f\) being a sum-of-squares of rational functions. Equivalently, there exists a nonzero polynomial
\(q\in\mathbb{R}[X_1,\ldots,X_n]\) such that \(q^2 f\) is a sum-of-squares of polynomials. Thus, compared with the form
\[
    \sum_{i=1}^r g_i^* f g_i \in \mathrm{SoS}+1
\]
in Condition (2) of the open problem, the form
\[
   0\neq \sum_{i=1}^r g_i^* f g_i \in \mathrm{SoS}
\]
may be viewed as a more natural noncommutative analog of Artin's solution.

By contrast, verifying membership in \(1+\mathrm{SoS}\) generally requires
additional structure or assumptions, and our present method does not seem to
provide an equivalent characterization of the original condition in the open
problem.  %For this reason, we formulate Theorem~\ref{main} in the above multiplier form.
\end{remark}

\begin{remark}
    Condition~(1) in Theorem~\ref{main} states that it is enough to consider self-adjoint tuple
$\underline{A}\in B(\mathcal{H})^m$ satisfying that $g(\underline A)e_1\neq0$ for every nonzero $g\in k\langle\underline X\rangle$ with $\deg(g)\le d$. For a fixed $d$, let $N(d)=1+m+\cdots+m^d$. For every Hilbert space $\mathcal H$ with $\dim\mathcal H>N(d)$, we can choose a family of mutually orthogonal vectors $\{e_\omega,|\deg(\omega)\le d\}$
in $\mathcal H$, indexed by all monomials of degree at most $d$. In particular, we may choose
 $e_{\emptyset}=e_1$, where $\emptyset$ denotes the empty monomial. For each
$i=1,\ldots,m$, define $L_i$ on the span of these vectors by
\[
L_i e_\omega=
\begin{cases}
e_{x_i\omega}, & \text{if } |\omega|<d,\\
0, & \text{if } |\omega|=d.
\end{cases}
\]
Extend $L_i$ to be zero on the orthogonal complement of this
subspace, and define
\[
B_i=L_i+L_i^*.
\]
Then each $B_i$ is a bounded self-adjoint operator on $\mathcal H$.
Moreover, the tuple $\underline{B}=(B_1,\ldots,B_m)$ 
satisfies the required condition. Indeed, for every monomial
\(\omega\) with \(|\omega|\leq d\),
\[
\omega(\underline{B})e_{\emptyset}
=
e_\omega+\text{terms of lower degree}.
\]
Therefore, the vectors
\[
\bigl\{\omega(\underline{B})e_{\emptyset}:|\omega|\leq d\bigr\}
\]
are linearly independent. For every self-adjoint tuple $\underline{A}=(A_1,\ldots,A_m)\in B(\mathcal H)^m$ 
and every $\varepsilon\in(0,1)$, define
\[
\underline{A}(\varepsilon)
=
(1-\varepsilon)\underline{A}
+
\varepsilon\underline{B}.
\]
Consider the Gram determinant
\[
p(\varepsilon)
=
\det
\left[
\left\langle
\omega(\underline{A}(\varepsilon))e_{\emptyset},
\nu(\underline{A}(\varepsilon))e_{\emptyset}
\right\rangle
\right]_{|\omega|,|\nu|\leq d}.
\]
Then $p(\varepsilon)$ is a polynomial in $\varepsilon$. Moreover, $p(1)\neq 0$, since \(\underline{A}(1)=\underline{B}\) satisfies the required condition. Hence $p$ is not the zero polynomial. Therefore, there exist arbitrarily small $\varepsilon>0$ such that $p(\varepsilon)\neq 0$.
Then for such $\varepsilon$, the vectors
\[
\left\{
\omega(\underline{A}(\varepsilon))e_{\emptyset}
:
|\omega|\leq d
\right\}
\]
are linearly independent. Therefore, the self-adjoint tuples satisfying the above condition are dense in $B(\mathcal H)_{\mathrm{sa}}^m$.

\end{remark}

  The following two sets are needed for the proof of Theorem~\ref{main}.
        \[\SoS^{d}\cyan{\coloneq}\{\sum_{i=1}^{r}g_i^*g_i| \deg(g_i)\leq d  {\text{ for every } 1\leq i\leq r}, r\in \mathbb{N}^{+}\}\]
        \[\Sigma_f^{d}\cyan{\coloneq}\{\sum_{i=1}^{r}g_i^*fg_i| \deg(g_i)\leq d {\text{ for every } 1\leq i\leq r}, r\in \mathbb{N}^{+}\}\]
        In this paper, the above definition is only applied to noncommutative symmetric polynomials $f$.
        
        For every integer $d\in\mathbb{N}$, the set $k\langle\underline{X}\rangle_d$ can be seen as the linear space spanned by all monomials whose degree is less than or equal to $d$, that is
        \[k\langle\underline{X}\rangle_d \coloneq  \operatorname{span}\{\omega|\deg(\omega)\leq d\},\]
        and the  dimension of $k\langle\underline{X}\rangle_d$ is 
        \[ N(d)=
\begin{cases}
d+1&\text{~for~} $m=1$,\\
\frac{m^{d+1}-1}{m-1}&\text{~for~} m>1.
\end{cases}
\]

The following are some results needed for the proof of Theorem~\ref{main}. In \cite{McCullough2005}, the original statements are proved for noncommutative polynomials in general noncommutative variables (not \cyan{necessarily} self-adjoint) over the complex field. As noted in the remark of \cite{McCullough2005}, the statements also hold in the real case and in the noncommutative symmetric variable case. \dede{After these adaptations,  Lemma~\ref{LemmaD} corresponds to  Proposition~3.4 in~\cite{McCullough2005}. For completeness, the proofs of these adapted results are included in the appendix.}

 \begin{lemma}\label{LemmaD}
    The set $\SoS^{d}$ is closed in the finite dimensional linear space $k\langle\underline{X}\rangle^{2d}$  \cyan{under the topology induced by the norm of noncommutative polynomial coefficients.} \dede{All norms on a finite-dimensional vector space are equivalent.}
\end{lemma}

\cyan{The following lemma is a modified analogue of \cite[Proposition 1.38]{Burg2016}.}
      \begin{lemma}\label{LemmaE}
            Let $f$ be a noncommutative symmetric polynomial and $d$ \cyan{be} a positive integer. Define $F=2 \deg(f) $.
            If there exists a vector $\underline{a}\in k^m$ \cyan{(regarded as $\mathrm{M_1}(k)^m$}) such that $f(\underline{a})>0$, then the set $\Sigma_f^{d}$ is closed in the finite \cyan{dimensional} linear space $k\langle\underline{X}\rangle^{2d+F}$
           \cyan{under the topology induced by the norm of noncommutative polynomial coefficients.}  \dede{All norms on a finite-dimensional vector space are equivalent.}      \end{lemma}
      \begin{proof}
      For $\underline{a}\in k^\cyan{m}$ and $\varepsilon>0$, define \[B(\underline{a}, \varepsilon) \coloneq \cup_{n\in \mathbb{N}}\{ \underline{A} \in S_n(k)^m \mid \| \underline{A} - I_n \otimes \underline{a} \| \le \varepsilon \} ,\]
      where $S_n$ denotes the \cyan{set of} self-adjoint matrices of dimension $n$, \cyan{and $I_n \otimes \underline{a}$=($a_1 I_n,\ldots, a_m I_n$)}. 
      Since there exists a vector $\underline{a}\in k^\cyan{m}$ such that $f(\underline{a})>0$, there exist  \cyan{$\varepsilon>0$ and $\delta>0$} such that  
      \[\cyan{ f(\underline{A})\succeq \delta I_n } \text{ for every }\underline{A}\in B(\underline{a}, \varepsilon).\]
      Let $p \in k\langle \underline{X} \rangle$, $\varepsilon > 0$, \cyan{if for all $\underline{A} \in B(\underline{a}, \varepsilon) $, the matrix $p(\underline{A})$ is zero}, then for all $n\in \mathbb{N},\underline{A} \in S_n(k)^m$, the matrix $p(\underline{A})$ is zero by the polynomial identity principle. Hence $p=0$ by Lemma~\ref{Lemmac} in Appendix.
      \cyan{Therefore, for such $\underline{a}$ and $\varepsilon>0$, we can introduce a norm on $k\langle\underline{X}\rangle^{2d+F}$ }
    \[ \|p\|_{\underline{a}}\coloneq\sup \{ \|p(\underline{A})\| | \underline{A} \in B(\underline{a}, \varepsilon) \} .\]
    It can be verified that this is a well-defined \cyan{norm}.
    Suppose a sequence of noncommutative polynomials
    $p_n \in \Sigma_f^{d}$ converges to $p \in k\langle\underline{X}\rangle^{2d+F}$. 
        Then $\|p_n\|_{\underline{a}}$ converges to $\|p\|_{\underline{a}}$, so that the sequence $\{\|p_n\|_{\underline{a}}\}$ is bounded. Since the set $\Sigma_f^{d}$ is a convex set in the real linear
        space $\Sym k\langle\underline{X}\rangle^{2d+F}$,
        by Carathéodory’s Theorem\cite{Reznick1992}, we can \cyan{assume} 
       \[p_n = \sum_{j=1}^{N(2d+F)+1} r_{j,n}^*f r_{j,n}.\]
     \cyan{As $f(\underline{A})\succeq \delta I_n $ for any $\underline{A}\in B(\underline{a}, \varepsilon)$, for each $j,n$, }
     $$  \cyan{ r_{j,n}(\underline{A})^*f (\underline{A})r_{j,n}(\underline{A})\succeq \delta r_{j,n}(\underline{A})^*r_{j,n}(\underline{A})} $$
$$  \cyan{ \|r_{j,n}^*f r_{j,n}\|_{\underline{a}} \geq \delta     \|r_{j,n}\|_{\underline{a}}^2.} $$
     \cyan{If  $\|p_n\|_{\underline{a}} \le M$, then} 
     $$ \delta  \|r_{j,n}(\underline{A})\|^2\le \|r_{j,n}^*f r_{j,n}\|_{\underline{a}} \leq M,$$
     \cyan{then} 
     \begin{eqnarray*}
    \|r_{j,n}(\underline{A})\|\le \sqrt{M/\delta },
\end{eqnarray*}
\cyan{and therefore}
\[ \|r_{j,n}\|_{\underline{a}}\le \sqrt{M/\delta }.\]

     By passing to a subsequence, there exists $r_j \in k\langle\underline{X}\rangle^d$ such that $r_{j,n}$ converges to $r_j$ for $j = 1, 2, \dots, N(2d+F)+1$ \cyan{under} the norm defined above.  Therefore, the sequence $\{p_n\}$ converges to a noncommutative polynomial $\cyan{p=}\sum^{N(2d+F)+1}_{j=1} r_j^*f r_j$ in $\Sigma_f^{d}$.
     \end{proof}

 The following are some results needed for the proof of Theorem~\ref{main} from \cite{Klee1955}. A $\phi-$cone in a topological linear space is a closed convex cone with origin vertex $\phi$; for a $\phi-$cone $A$, let
     \[
      A' := A \cap (-A)
     \]
     denote the linear subspace contained in $A$.
     \begin{lemma}\cite[Theorem 2.5]{Klee1955}\label{LemmaS}
         Suppose $A$ and $B$ are $\phi-$cones in a locally convex topological linear space $E$, $A$ is locally compact, and $A \cap B = \{0\}$. Then $E$ admits a continuous linear functional $f$ such that
\[
f < 0 \text{ on } A \setminus A', \quad
f = 0 \text{ on } A' \cup B', \quad
f \ge 0 \text{ on } B \setminus B'.
\]
\cyan{Moreover, if $B$ is also locally compact, then  $E$ admits a continuous linear functional $f$ such that}
\[
f < 0 \text{ on } A \setminus A', \quad
f = 0 \text{ on } A' \cup B', \quad
f > 0 \text{ on } B \setminus B'.
\]

\end{lemma}

 \begin{lemma}\label{LemmaF}
Let $d$ be an integer and let $f$ be a noncommutative symmetric polynomial. Then the following conditions are equivalent.
\begin{enumerate}[label=(\arabic*)]
            \item  For any  nontrivial separable $k$-Hilbert space $\mathcal{H}$ and every self-adjoint tuple $\underline{A}\in B(\mathcal{H})^m$ and nonzero vector $v\in \mathcal{H}$, if $g(\underline A)v\neq0$ for every nonzero $g\in k\langle\underline X\rangle$ with $\deg(g)\le d$,  the compression of $f(\underline{A})$ on $W_{(\underline{A},v)}$ is not negative definite, where the subspace $W_{(\underline{A},v)}:= \spans \{\omega(\underline{A}) v|\deg(\omega) \leq d\}$;
              \item  For any  nontrivial separable $k$-Hilbert space $\mathcal{H}$ with a fixed unit vector $e_1$ and every self-adjoint tuple $\underline{A}\in B(\mathcal{H})^m$ , if $g(\underline A)e_1\neq0$ for every nonzero $g\in k\langle\underline X\rangle$ with $\deg(g)\le d$,  the compression of $f(\underline{A})$ on $W_{\underline{A}}$ is not negative definite, where  the subspace $W_{\underline{A}}:= \spans \{\omega(\underline{A}) e_1| \\ \deg(\omega) \leq d\}$;
        \end{enumerate}
      \end{lemma} 
    \begin{proof}

The implication (1)$\Rightarrow$(2) is obvious.\\
(2)$\Rightarrow$(1). 
For any nontrivial separable $k$-Hilbert space $\mathcal{H}$ with standard orthonormal basis $\{e_1,e_2,\ldots\}$ and any nonzero vector $v\in\mathcal{H}$, there exists a unitary operator $U_v$ such that 
\[\frac{v}{\|v\|}=U_v e_1,\]
For any self-adjoint tuple $\underline{A}\in B(\mathcal{H})^m$, define 
\[ U_v^* \underline{A} U_v=(U_v^* A_1 U_v,\ldots,U_v^* A_m U_v).\]
Then for every nonzero $g\in k\langle\underline X\rangle$ with $\deg(g)\le d$
\[g(\underline A)v\neq0 \Rightarrow g(U_v^* \underline{A} U_v)e_1\neq0. \]
\[
\begin{aligned}
W_{U_v^* \underline{A} U_v}=W_{(U_v^* \underline{A} U_v,e_1)}:&= \spans \{\omega(U_v^* \underline{A} U_v) e_1|\deg(\omega) \leq d\} \\
  &= \spans \{U_v^*\omega( \underline{A})  U_v e_1|\deg(\omega) \leq d\}\\
  &=\spans \{U_v^*\omega( \underline{A}) v|\deg(\omega) \leq d\}\\
  &=U_v^*W_{(\underline{A},v).}
\end{aligned}\]
Condition (2) implies that the compression of $f(U_v^* \underline{A} U_v)$ on $W_{U_v^* \underline{A} U_v}$ is not negative definite. By the following equation
\[
f(U_v^* \underline{A} U_v)= U_v^* f(\underline{A}) U_v
\]
it can be concluded that the compression of $f(\underline{A})$ on $W_{(\underline{A},v)}$ is not negative definite.
    \end{proof}
Everything is now in place to prove Theorem~\ref{main}.

   \begin{proof}[Proof of  Theorem~\ref{main}]

            (2)$\Rightarrow$(1). Let 
            \[d=\max\{\deg(g_1),\deg(g_2),\ldots,\deg(g_r)\}.\]
        For any  nontrivial $k$-Hilbert space $\mathcal{H}$ with standard orthonormal basis $\{e_1,e_2,\ldots\}$ and self-adjoint tuple $\underline{A}\in B(\mathcal{H})^m$, 
        \[\sum_{i=1}^{r}g^*_i(\underline{A}) f(\underline{A}) g_i(\underline{A})\succeq 0.\]
        Then  we have
        \begin{eqnarray*}
         e_1^*\sum_{i=1}^{r}g^*_i (\underline{A}) f(\underline{A}) g_i(\underline{A})e_1&=&\sum_{i=1}^{r}(g_i (\underline{A})e_1)^* f(\underline{A}) g_i(\underline{A})e_1 , 
        \end{eqnarray*}
        If $g(\underline A)e_1\neq0$ for every nonzero $g\in k\langle\underline X\rangle$ with $\deg(g)\le d$, nonzero vector
        $g_i(\underline{A})e_1$ belongs to the subspace $W_{\underline{A}}$, \cyan{which implies} that the compression of $f(\underline{A})$ on $W_{\underline{A}}$ is not negative definite.\\
        \cyan{ To prove (1)$\Rightarrow$(2), it suffices to prove $\neg (2) \Rightarrow \neg (1)$.} \dede{The negation of condition~(2) can be equivalently stated as follows.}\\
        \textit{ $\neg(2)$ For any $d\in \mathbb{N}$, there do not exist $r \in \mathbb{N}^{+} $ and  noncommutative polynomials $g_1,\dots g_r\in k\langle\underline{X}\rangle$  such that 
            \[\max\{\deg(g_1),\deg(g_2),\ldots,\deg(g_r)\}\leq d,\]
			\begin{equation*}
				0\neq\sum_{i=1}^r{g_i^* f g_i}\in{\rm SoS}.
			\end{equation*}}
For a fixed $d\in \mathbb{N}$, the above statement  means $\SoS^{d+1+F/2} \cap \Sigma^{d}_f=\{0\}$.

        By Lemma~\ref{LemmaD}, the set $\SoS^{d+1+F/2}$ is a closed cone \cyan{in} $k\langle\underline{X}\rangle^{2d+2+F}$. By assumption in Theorem~\ref{main}, there exists a vector $\underline{a}\in k^m$, such that $f(\underline{a})>0$. Then by Lemma~\ref{LemmaE}, the set $\Sigma^{d}_f $ is a closed convex cone. 
\dede{In fact, both cones are contained in the real vector space $\Sym k\langle\underline{X}\rangle^{2d+2+F}$, which consists of symmetric noncommutative polynomial in $k\langle\underline{X}\rangle^{2d+2+F}$.
}
        
        By Lemma~\ref{LemmaS}, \cyan{both $\SoS^{d+1+F/2}$ and $\Sigma^{d}_f$} are $\phi$-cones in a locally convex topological vector space $k\langle\underline{X}\rangle^{2d+2+F}$. Since  $\SoS^{d+1+F/2}$ and $\Sigma^{d}_f$ are closed in this finite-dimensional space, they are locally compact. Then there exists a continuous linear functional \dede{ $\widetilde\mu:\Sym k\langle\underline{X}\rangle^{2d+2+F}\rightarrow \mathbb{R}$} such that 
        \[\widetilde\mu(\Sigma^{d}_{f}\setminus \{0\} ) < 0 , \quad
        \widetilde\mu(\SoS^{d+1+F/2}\setminus \{0\}) > 0 \]
       \dede{ When $k=\mathbb{R}$, then
$\widetilde\mu$ can be extended to a real linear functional on $\mathbb{R}\langle\underline{X}\rangle^{2d+2+F}$ by
\[
 \mu(f)
=
\widetilde\mu\left(\frac{f+f^*}{2}\right),
\qquad f\in \mathbb{R}\langle\underline{X}\rangle^{2d+2+F}.
\]
When $k=\mathbb{C}$, $\widetilde\mu$ has a unique complex linear extension to $\mathbb{C}\langle\underline{X}\rangle^{2d+2+F}$, defined by
\[
 \mu(f)
=
\widetilde \mu\left(\frac{f+f^*}{2}\right)
+
i\widetilde\mu\left(\frac{f-f^*}{2i}\right),\qquad f\in \mathbb{C}\langle\underline{X}\rangle^{2d+2+F}.
\]}
By the proof of Lemma~\ref{lemmaB}, there exists a Hilbert space $\mathcal{K}$ of dimension $N(d+F/2)$, a vector $\hat{1} \in \mathcal{K}$, and a self-adjoint operator tuple $\underline{\mathcal{X}}\in B(\mathcal{K})^m$ such that
        \[p(\underline{\mathcal{X}})\hat{1} = p,\]
        and 
      \[
       \langle p(\underline{\mathcal{X}})\hat{1},\, q(\underline{\mathcal{X}})\hat{1} \rangle = \dede{ \mu}(q^*p)
      \]
      for all $p, q\in k\langle\underline{X}\rangle^{d+F/2}$.
     Define the linear subspace 
     \[W_{(\mathcal{\underline{X}},\hat{1})}:=\spans\{\omega(\mathcal{\underline{X}})\hat{1}|\deg(\omega) \leq d\}.\]
     Moreover, for every nonzero polynomial
$g\in k\langle\underline{X}\rangle$ with $\deg(g)\leq d$, we have
\[
\|g(\mathcal{\underline{X}})\hat{1}\|^2
=
\dede{ \mu}(g^*g)>0,
\]
which implies that
$g(\mathcal{\underline{X}})\hat{1}\neq0$.
     Furthermore, for every nonzero vector $\eta\in W_{(\mathcal{\underline{X}},\hat{1})}$, there exists a noncommutative polynomial $g$ whose degree is \cyan{at most} $d$ such that $\eta= g(\mathcal{\underline{X}})\hat{1} $, then

\[
f(\mathcal{\underline{X}})g(\mathcal{\underline{X}})\hat{1}=(fg)(\mathcal{\underline{X}})\hat{1},
\qquad
\deg(fg)\leq d+F.
\]

\[
\left\langle f(\mathcal{\underline{X}})g(\mathcal{\underline{X}})\hat{1},\; g(\mathcal{\underline{X}})\hat{1} \right\rangle
=
\mu\!\left(g^{*}fg\right)< 0,
\]
    which shows that the compression of $f(\mathcal{\underline{X}})$ on $W_{(\mathcal{\underline{X}},\hat{1})}$ is negative definite. Therefore, by Lemma~\ref{LemmaF}, condition~(1) does not hold.

      \end{proof}

\begin{remark}
Condition~(1) in Theorem~\ref{main} should not be understood as a stronger version of Condition~(1) in the open problem of~\cite{IgorMarkus2007nicht}; however, Condition (2) of the open problem is stronger than Condition (2) of
Theorem~\ref{main}, since \(1+{\rm SoS}\subseteq {\rm SoS}\). 
\end{remark}

% \section{\cyan{Neither $\RSoS$ nor Its Complement is Convex}}
% In contrast to \cyan{$\SoS$}, \cyan{$\RSoS$} is \cyan{non-convex, as \dede{it will be} demonstrated by constructing explicit noncommutative polynomials $f_1$ and $f_2$, and showing that $f_1,f_2\in\RSoS$ but $1/2(f_1+f_2)\notin \RSoS$.}

\section{Neither $\RSoS$ nor Its Complement Is Convex}

In contrast to $\SoS$, which is a convex cone, $\RSoS$ is not convex when
the free algebra has at least three generators. We prove this by constructing
explicit symmetric noncommutative polynomials $f_1,f_2\in\RSoS$ such that
\[
    \frac{1}{2}(f_1+f_2)\notin\RSoS.
\]
We further show that the complement of $\RSoS$ is also not convex.

 \begin{theorem}\label{Thm:RSOS-notcvx}
  % The set $\RSoS$ for the free associative algebra $k\langle X_1,\ldots,X_m\rangle$  with  at least $m \geq 3$ generators is \textbf{not convex}.
  For the free associative algebra $k\langle X_1,\ldots,X_m\rangle$ with $m \geq 3$, the set $\RSoS$ is {not convex}.
 \end{theorem}
\begin{comment}
\begin{theorem}\label{Thm:RSOS-notcvx}
Let \(m\ge 3\). Then the set \(\RSoS\) in the free algebra
\(k\langle X_1,\ldots,X_m\rangle\) is not convex.
\end{theorem}   
\end{comment}

\begin{proof}
Define $f_1$ and $f_2$ in $k\langle X,Y,Z\rangle$  as 
\begin{eqnarray*}
    &f_1\coloneq2Y^2X^2Y^2+Y^2X^2Y+YX^2Y^2+X^2,\\
    &f_2\coloneq2Z^2X^2Z^2+Z^2X^2Z+ZX^2Z^2+X^2,
\end{eqnarray*}
% \dede{It will be proved} that $f_1$ and $f_2$ are both in $\RSoS$, but $f_1+f_2\notin \RSoS$.  

We prove that $f_1,f_2\in\RSoS$, whereas $f_1+f_2\notin\RSoS$. Since
$\RSoS$ is invariant under multiplication by positive scalars,
\[
    \frac{1}{2}(f_1+f_2)\notin\RSoS.
\]
Therefore, $\RSoS$ is not convex in any free algebra with at least three
generators.

By Theorem~\ref{definition}, $f_1,f_2$ are noncommutative $\RSoS$ polynomials since
\begin{eqnarray*}
  f_1+Yf_1Y&\in\SoS,\\
  f_2+Zf_2Z&\in \SoS.
\end{eqnarray*}
Define
    \[f_3\coloneq f_1+f_2.\]
We will  prove that \cyan{$f_3 \notin \RSoS$}, equivalently, there do not exist finitely many noncommutative polynomials $\{g_i\}_{i=1}^r$ and $\{h_j\}_{j=1}^t$ such that
\[\sum_{i=1}^{r}g_i f_3 g_i^*=\sum_{j=1}^{t} h_jh_j^*. \]
% \cyan{In this proof, we use the equivalent rational
% sum-of-squares form above instead of $\sum_{j=1}^{t} h_j^*h_j$ for clarity when counting and comparing variables.}
\cyan{In this proof, we use the equivalent formulation of rational sums of squares given above, with $\sum_{i=1}^{r}g_i f_3 g_i^*$ in place of $\sum_{i=1}^{r}g_i^* f_3 g_i$, since it facilitates the counting and comparison of variables.}

The statement is proved by contradiction.
For any integer $r\in\mathbb{N}^{+}$ and noncommutative polynomials $g_1,\ldots,g_r$, 
   \[g_i=\sum_{k=1}^{N_i}\alpha_{ik}\omega_{ik},~\text{where}~1\leq i\leq r~\text{and}~N_i\in \mathbb{N}.\]

   Define $F=\sum_{i}g_i\dede{f_3}g_i^*$ and $D\coloneqq \max_{1\le i\le r} \deg(g_i)$. 
    Without loss of generality, suppose $D\geq 1$.   
     \cyan{Indeed, when $D=0$, there exists $c> 0$ and $\{h_j\}_{j=1}^t$  such that
     \[cf_3=\sum_{j=1}^{t} h_jh_j^*,\]
     therefore
     \[Xf_3X^*=\frac{1}{c}\sum_{j=1}^{t} Xh_j(Xh_j)^*,\]
  where $X^*=X$, which  reduces  to the case where $r=1$, $g_1=X$ and $D=1$.
    For $D \geq 1$,} there exists a monomial $\omega$ in $g_i$ for some $1\leq i \leq r$ such that $\deg(\omega)=D$. The last variable of the word $\omega$ can be $X$, $Y$, or $Z$. \cyan{We first consider the case in which the final letter of $\omega$ is $Y$,} and thus $\omega=uY$ for some monomial $u$ with $\deg(u)=D-1$.
    
For any  $d \in \mathbb{N}$, $p\in  \mathbb{C}\langle X,Y,Z\rangle$, let $p^{(d)}$ be the set of words appearing (that is, with non-zero coefficient) in $p$ of degree exactly $d$, i.e.,
\[ p^{(d)}\coloneq\{v\in\langle X,Y,Z\rangle| \deg (v)=d, \text{ coefficient of } v \text{ is non-zero}\}, \]
and $ p^{(d)*}$  is the set $\{v^*| v \in p^{(d)}\}$. If $S_1$ and $S_2$ 
  are two sets of words, then  $S_1S_2$  is defined as the set $\{v_1v_2| v_1 \in S_1,v_2 \in S_2\}$.
  
  Next, assume that  $\sum_{i}g_i\dede{f_3}g_i^*=\sum_{j}h_jh_j^*.$ 

\paragraph{Step 1. \cyan{Prove  that  every word in $h_j^{(D+3)}$ must end with $Y^2X$ or $Z^2X$.}
}
Since $F=\sum_{i}g_if_3g_i^*$, with $\deg(F)=2D+6$, 
\begin{equation*}
    F^{(2D+6)}\subseteq \{v_1 Y^2X^2Y^2 v_2^* |  v_1,v_2  \in \cup_i g^{(D)}_i\}\cup \{v_1 Z^2X^2Z^2 v_2^*| v_1,v_2 \in \cup_i g^{(D)}_i\}.
\end{equation*}

As $F=\sum_j h_jh_j^*$, we know that words in $h_j^{(D+3)}$ must end with  $Y^2X$ or $Z^2X$, i.e.,
\begin{equation*}
h_j^{(D+3)}\subseteq \{vY^2X |\deg(v)=D \}\cup \{vZ^2X |\deg(v)=D\}, \forall j \in \dede{\{1,2,\ldots,t\}}.
\end{equation*}
Note that there is no word ending with $YZX$ that appears in $h_j^{(D+3)}$.

\paragraph{Step 2. Prove  that $uYZX$ appears in $\cup_jh_j^{(D+2)}$.}
Consider the word $uYZX^2Z^2Yu^*$ and
\begin{equation*}
F^{(2D+5)}\subseteq  \bigcup_{\cyan{i}} \left( g_i^{(D)}f_3^{(5)}g_i^{(D)*} \cup  g_i^{(D-1)}f_3^{(6)}g_i^{(D)*} \cup g_i^{(D)}f_3^{(6)}g_i^{(D-1)*} \right).
\end{equation*}
Note that by the construction of $f_3$, the variables at positions  $D+1$, $D+2$, $D+3$ of words in  $ g_i^{(D)}f_3^{(6)}g_i^{(D-1)*}$ are always $Z^2X$ or  $Y^2X$ and cannot be $YZX$, which is similar to $g_i^{(D-1)}f_3^{(6)}g_i^{(D)*} $.
$$uYZX^2Z^2Yu^*\notin  \bigcup_{i}  \left(g_i^{(D-1)}f_3^{(6)}g_i^{(D)*} \cup g_i^{(D)}f_3^{(6)}g_i^{(D-1)*} \right).$$ 
Hence, $uYZX^2Z^2Yu^* \in  \bigcup_{i} g_i^{(D)}f_3^{(5)}g_i^{(D)*}$ and it always appears with positive coefficients, as $uY$ appears in some $g_i^{(D)}$. Moreover,  the coefficient of 
$uYZX^2Z^2Yu^*$ in $\sum_i g_if_3g_i^*$ is non-zero.

Now, \dede{consider the word}
\begin{equation*}
    uYZX^2Z^2Yu^* \in \left(\bigcup_{j} h_j^{(D+2)} h_j^{(D+3)*} \cup h_j^{(D+3)} h_j^{(D+2)*} \right).
\end{equation*}

\begin{enumerate}
\item If $ uYZX^2Z^2Yu^* =\omega_1\omega_2 $ for some $ \omega_1\in  h_j^{(D+3)}$ and $  \omega_2 \in h_j^{(D+2)*}$ for some $j$, then  $\omega_1$ must end with $Y^2X$ or $Z^2X$, and then the variables at positions  $D+1$, $D+2$, $D+3$ of the word $uYZX^2Z^2Yu^*$ must be $Y^2X$ or $Z^2X$ respectively, which is not true.
\item If $ uYZX^2Z^2Yu^* =\omega_1\omega_2 $ with some $ \omega_1\in  h_j^{(D+2)}$ and $  \omega_2 \in h_j^{(D+3)*}$ for some $j$, then the word $\omega_1=uYZX \in  h_j^{(D+2)}$.
\end{enumerate}

\paragraph{Step 3. Prove  that the  word $uYZX^2ZYu^*$  appears in $\sum_jh_jh_j^*$}
Now, consider the word $uYZX^2ZYu^*$, 
$$\deg( uYZX^2ZYu^*)=2D+4,$$
it remains to consider the following sets
\begin{equation*}
    h_j^{(D+2)}h_j^{(D+2)*},  h_j^{(D+3)}h_j^{(D+1)*} ,  h_j^{(D+1)}h_j^{(D+3)*}.
\end{equation*}
The variables at positions  $D+1$,$D+2$,$D+3$ of the  words in $h_j^{(D+3)}h_j^{(D+1)*}$  must be $Z^2X$ or $Y^2X$, thus $uYZX^2ZYu^*$ is not in this set, and  similarly $uYZX^2ZYu^* \notin h_j^{(D+1)}h_j^{(D+3)*}$. \dede{It has already been proved that} $uYZX \in  h_j^{(D+2)}$ for some $j$.
Hence, the coefficients of the   word $uYZX^2ZYu^*$ from $ h_j^{(D+2)}h_j^{(D+2)*}$ in the polynomial $\sum_j h_jh_j^*$ must be \dede{positive}, \cyan{
indeed, the only decomposition of the 
word  $uYZX^2ZYu^*$ in $h_j^{(D+2)}h_j^{(D+2)*}$  must be in the form of  $uYZX(uYZX)^*$.}

\paragraph{Step 4. Prove  that the word $uYZX^2ZYu^*$ does not appear in $\sum_ig_if_3g_i^*$}

Now, consider the word $uYZX^2ZYu^*$ as well as the set $F^{(2D+4)}$; note that $f_3^{(4)}$ is empty, therefore
\begin{eqnarray*}
    F^{(2D+4)}\subseteq && \bigcup_{i} \left( g_i^{(D-1)}f_3^{(6)}g_i^{(D-1)*} \cup  g_i^{(D-1)}f_3^{(5)}g_i^{(D)*} \cup g_i^{(D)}f_3^{(5)}g_i^{(D-1)*}\right.\\
    && \left.\cyan{\cup g_i^{(D)}f_3^{(6)}g_i^{(D-2)*} \cup g_i^{(D-2)}f_3^{(6)}g_i^{(D)*}} \right).
\end{eqnarray*}
Then, we have
\begin{enumerate}
\item  The variables at positions  $D$, $D+1$, $D+2$ of the   words in $g_i^{(D-1)}f_3^{(6)}g_i^{(D-1)*}$  must \dede{appear in} $f_3^{(6)}$ and must be $Z^2X$ or $Y^2X$, and cannot be $YZX$, thus the word $uYZX^2ZYu^*$ is not in $g_i^{(D-1)}f_3^{(6)}g_i^{(D-1)*}$. 
\item  For words $\omega'$ in $g_i^{(D-1)}f_3^{(5)}g_i^{(D)*}$ and $g_i^{(D)}f_3^{(5)}g_i^{(D-1)*}$, at least one of the conditions is satisfied:
\begin{itemize}
    \item The variables at positions \cyan{$D$, $D+1$, $D+2$, $D+3$ of $\omega'$ must be $Z^2X^2$, $ZX^2Z$, $YX^2Y$ or $Y^2X^2$, which comes from  words in  $f_3^{(5)}$.}; it cannot be $YZX^2$; thus, $uYZX^2ZYu^*$ is not in these sets.
    \item The variables at positions \cyan{$D+2$,  $D+3$, $D+4$, $D+5$ of $\omega'$ must be either  $X^2Z^2$,$ZX^2Z$, $YX^2Y$ or $X^2Y^2$.} cannot be $X^2ZY$; thus, $uYZX^2ZYu^*$ is not in these sets.  
\end{itemize}
\item  \cyan{The variables at positions  $D+1$, $D+2$, $D+3$ of the   words in   $g_i^{(D)}f_3^{(6)}g_i^{(D-2)*}$,  must appear in $f_3^{(6)}$ and must be either $Z^2X$ or $Y^2X$, and cannot be $ZX^2$, thus the word $uYZX^2ZYu^*$ is not in $g_i^{(D)}f_3^{(6)}g_i^{(D-2)*}$. However, the letters at these positions in
$uYZX^2ZYu^* $ are $ZX^2$. Similarly, the term $g_i^{(D-2)}f_3^{(6)}g_i^{(D)*}$ is excluded analogously by reading the variables from right to left.}
\end{enumerate}
So  $uYZX^2ZYu^*$ does not appear in $\sum_i g_if_3g_i^*$, but it appears in $\sum_j h_jh_j^*$, a contradiction.

\cyan{If the last letter of  $\omega$ is not $Y$, denote it by $T$. We then track the coefficient of the word $uTYX^2YTu^*$. 
Repeating Steps 2-4 with
$uTYX^2Y^2Tu^*$  and 
$uTYX^2YTu^*$
in place of
$uYZX^2Z^2Yu^*$ and $uYZX^2ZYu^*$
gives the same contradiction. This statement is also suitable for the case of more than three generators, since $T$ may be any generator other than $Y$.
}

\end{proof}

In the commutative case, by Artin's solution to Hilbert's seventeenth problem, the set of polynomials that can be written as sums of squares of rational functions is exactly the set of nonnegative polynomials. Hence this set is convex. In contrast, in the noncommutative case, $\RSoS$ is not convex. This shows a fundamental difference between the commutative and noncommutative cases, and explains why Theorem~\ref{main} has a more complicated form.

% \begin{theorem}\label{Thm:RSOSc-notcvx}
%  The set $\RSoS^c$ for the free algebra \cyan{is \textbf{not convex}}.
% \end{theorem}
% \begin{proof}
% \cyan{Let $X$ be a generator of  the given free algebra. Define 
% $f_1\coloneq X$ and $f_2 \coloneq -X$, then $f_1$ and $f_2$ are noncommutative homogeneous polynomials, and clearly $f_1,f_2 \notin \SoS$.  By Theorem~\ref{thm:homogeneous:RSOS=SOS}, $f_1,f_2 \notin \RSoS$. However, $f_1+f_2 =0 \in \RSoS$, hence  $\RSoS^c$  is not convex.
% }
% \end{proof}

\begin{theorem}\label{Thm:RSOSc-notcvx}
Let $k\langle X_1,\ldots,X_m\rangle$ be the free associative algebra over $k$ with
$m\geq 1$. Then the complement $\RSoS^c$ in the vector space $\cyan{\Sym} k\langle X_1,\ldots,X_m\rangle$ is not convex. 
\end{theorem}

\begin{proof}
Let $X$ be one of the generators and define
\[
    f_1\coloneq X,
    \qquad
    f_2\coloneq -X.
\]
Both $f_1$ and $f_2$ are symmetric homogeneous polynomials of degree one.
Since every nonzero sum of Hermitian squares has even degree, neither
$f_1$ nor $f_2$ belongs to $\SoS$. By
Theorem~\ref{thm:homogeneous:RSOS=SOS}, it follows that
\[
    f_1,f_2\notin\RSoS.
\]
However,
\[
    \frac{f_1+f_2}{2}=0\in\RSoS.
\]
Thus, two elements of $\RSoS^c$ have a midpoint outside $\RSoS^c$, and
therefore $\RSoS^c$ is not convex.
\end{proof}

\begin{comment}

\begin{remark}
Similar to Theorem~\ref{Thm:RSOS-notcvx}, Theorem~\ref{Thm:RSOSc-notcvx} also corresponds to the nonexistence of some equivalent forms.

Let $\underline{X} = (X_1, X_2, \ldots, X_m)$ be generators with $m \geq 2$. 
Then there \textbf{do not exist} families of linear spaces $\{V_i\}_{i \in I}$, 
convex sets $\{C_i \subseteq V_i\}_{i \in I}$, and linear maps 
$\{\psi_i: \mathbb{C}\langle \underline{X}\rangle \to V_i\}_{i \in I}$ 
such that the following equivalence holds for all $f \in \mathbb{C}\langle \underline{X}\rangle$:
\[
f \in \RSoS \quad \Longleftrightarrow \quad \exists   i \in I \text{ such that  }  \psi_i(f) \notin C_i.
\]
\end{remark}

\end{comment}

\bibliographystyle{plain}
\bibliography{optimization}

\appendix

\section*{Appendix}
\setcounter{theorem}{0}
\renewcommand{\thetheorem}{A.\arabic{theorem}}
        \begin{lemma}\label{lemmaA}
            There exists a linear functional $\mu : k\langle\underline{X}\rangle_{2d} \to k$ such that $\mu(\SoS^{d}\setminus0) > 0$.
        \end{lemma}

\begin{proof}
            Proof by induction.  $\mu_0 : k\langle\underline{X}\rangle_{0}   \cong  k \to k$ defined by 
\[
\mu_0(a) = a.
\]
Suppose $\mu_d : k\langle\underline{X}\rangle_{2d} \to k$ is a linear functional such that
\[
\mu_d(\SoS^{d}\setminus0) > 0.
\]
Define an extension $\mu_{d+1} : k\langle\underline{X}\rangle_{2d+2} \to k$ of $\mu_d$, as follows.
\[
\begin{cases}
\mu_{d+1}(\omega)=\mu_d(\omega) &\text{ if } \deg (\omega) \le 2d, \\
\mu_{d+1}(\omega)=c &\text{ if } \omega=v^*v \text{, where }\deg (v) = d+1, \\
\mu_{d+1}(\omega) = 0 &\text{ otherwise. }\\
\end{cases}
\]
Since the form $\langle p, q \rangle_d = \mu_d(q^*p)$ is (strictly) positive definite on $k\langle\underline{X}\rangle_{d}$, there exists a real number $c > 0$ such that $\langle p, q\rangle_{d+1} = \mu_{d+1}(q^*p)$
is positive definite on $k\langle\underline{X}\rangle_{d+1}$.
\end{proof}

\begin{lemma}\label{lemmaB}
    Suppose $\mu : k\langle\underline{X}\rangle_{2d+2} \to k$ is a linear functional. If $\mu(\SoS^{d+1}\setminus0) > 0$, then there exists a Hilbert space $\mathcal{K}$ of dimension $N(d)$, a vector $\gamma \in \mathcal{K}$, and a self-adjoint operator tuple $\underline{\mathcal{X}}\in B(\mathcal{K})^m$ such that
\[
\langle p(\underline{\mathcal{X}})\gamma,\, q(\underline{\mathcal{X}})\gamma \rangle = \mu(q^*p)
\]
for all $p, q\in k\langle\underline{X}\rangle_{d}$.
\end{lemma}

\begin{proof}
        Let $\mathcal{K'}$ denote the Hilbert space obtained by defining the inner product $\langle p, q \rangle = \mu(q^* p)$ on $k\langle\underline{X}\rangle_{d+1}$ as in the GNS construction.
        The \cyan{condition of}  $\mu$ guarantees that there are no null vectors. Let $\mathcal{K}$ denote the span of $k\langle\underline{X}\rangle_d$ in $\mathcal{K'}$, and let $\mathcal{N}$ denote its orthogonal complement. Define the operators $S_j$ by $S_j p = X_j p$  for all $p \in \mathcal{K}$, and $S_j p = 0$ for all $p \in \mathcal{N}$. Let $P_{\mathcal{K}}$ denote the orthogonal projection of $\mathcal{K'}$ onto $\mathcal{K}$, 
        and define $\mathcal{X}_j = P_{\mathcal{K}} S_j P_{\mathcal{K}}$.\\
        It follows that for $p, q \in \mathcal{K}$, 
        \[\langle \mathcal{X}_j p, q \rangle = \langle S_j p, q \rangle = \langle X_j p, q \rangle = \langle p, X_j q \rangle = \langle p, \mathcal{X}_j q \rangle,\]
       Thus, $\mathcal{X}_j = \mathcal{X}_j^*$, and therefore $p(\underline{\mathcal{X}})\hat{1} = p$ for all $p \in k\langle\underline{X}\rangle_d$. 
        Furthermore, if $q \in k\langle\underline{X}\rangle^d$ as well, then $\langle p(\underline{\mathcal{X}})\hat{1}, q(\underline{\mathcal{X}})\hat{1} \rangle = \langle p, q \rangle$.
\end{proof}

\begin{lemma}\label{Lemmac}
    There exists a tuple $\underline{\mathcal{X}} = (\mathcal{X}_1, \ldots, \mathcal{X}_m)$ from $M_{{N(d)}}(k)$ such that if $p \in k\langle\underline{X}\rangle_d$ and $p(\mathcal{X}) = 0$, then $p = 0$.
\end{lemma}

\begin{proof}
   By Lemma~\ref{lemmaA} and Lemma~\ref{lemmaB}.
\end{proof}

\begin{proof}[Proof of Lemma~\ref{LemmaD}]
         By Lemma~\ref{Lemmac}, there exists a tuple $\underline{\mathcal{X}} = (\mathcal{X}_1, \ldots, \mathcal{X}_m)$ from $B(k^{N(2d)})$ such that if $p \in k\langle\underline{X}\rangle_{2d}$ and $p(\underline{\mathcal{X}}) = 0$, then $p = 0$.\\
        \cyan{Define $\|p\|_{\underline{\mathcal{X}}} \coloneq \|p(\underline{\mathcal{X}})\|$, then $\|p\|_{\underline{\mathcal{X}}}$ is a norm on $k\langle\underline{X}\rangle_{2d}$}.
        
        Suppose a sequence of noncommutative $\SoS^d$ polynomials $\{p_n \}\in k\langle\underline{X}\rangle_{2d}$ converges to $p \in k\langle\underline{X}\rangle_{2d}$. 
        Then $p_n(\underline{\mathcal{X}})$ converges to $p(\underline{\mathcal{X}})$, so that the sequence $\{p_n(\underline{\mathcal{X}})\}$ is bounded. Since the set $\SoS^d$ is a convex set in the real linear
        space $\Sym k\langle\underline{X}\rangle_{2d}$,
        by Carathéodory’s Theorem\cite{barvinok2025course}, we can \cyan{assume} 
       \[p_n = \sum_{j=1}^{N(2d)+1} r_{j,n}^* r_{j,n}.\]
       Evaluating this \cyan{SoS} form at $\underline{\mathcal{X}}$, one concludes that each sequence $\{r_{j,n}(\underline{\mathcal{X}})\}_n$ is bounded, and thus, by passing to a subsequence, \cyan{for $j = 1, 2, \dots, N(2d)+1$}, there exists $r_j \in k\langle\underline{X}\rangle_d$ such that $r_{j,n}(\underline{\mathcal{X}})$ converges to $r_j(\underline{\mathcal{X}})$. 
      Therefore, the sequence $\{p_n\}$ converges to a noncommutative $\SoS$ polynomial $\sum^{N(2d)+1}_{j=1} r_j^* r_j$.
      \end{proof}
\end{document}